\documentclass[11pt]{amsart}
\usepackage{t1enc}
\usepackage[latin1]{inputenc}
\usepackage{graphicx}
\usepackage{amsthm}
\usepackage{amsmath}
\usepackage{amssymb}
\usepackage{sseq}
\usepackage[pagebackref]{hyperref}
\usepackage{hyperref}
\usepackage[curve]{xypic}
\usepackage[left=3.2cm,right=3.2cm,top=3cm,bottom=3cm]{geometry} 
\usepackage{stmaryrd}
\usepackage{marvosym}
\usepackage{tikz}

\newtheorem{thm}{Theorem}[section]
\newtheorem*{thm*}{Theorem}
\newtheorem*{mthm*}{Main Theorem}

\newtheorem{lemma}[thm]{Lemma}

\newtheorem{cor}[thm]{Corollary}

\newtheorem{prop}[thm]{Proposition}

\newtheorem*{conjecture*}{Conjecture}

\newtheorem*{question*}{Question}

\theoremstyle{definition}
\newtheorem{defi}[thm]{Definition}

 \newtheorem{example}[thm]{Example}
  \newtheorem*{example*}{Example}

\theoremstyle{remark}

\newtheorem{remark}[thm]{Remark}

\newcommand{\Ac}{\mathcal{A}}
\newcommand{\BB}{\mathcal{B}}
\newcommand{\CC}{\mathcal{C}}
\newcommand{\DD}{\mathcal{D}}
\newcommand{\EE}{\mathcal{E}}
\newcommand{\FF}{\mathcal{F}}

\newcommand{\PP}{\mathcal{P}}

\newcommand{\KK}{\mathcal{K}}

\newcommand{\MM}{\mathcal{M}}
\newcommand{\NN}{\mathcal{N}}

\DeclareMathOperator{\Sd}{Sd}

\DeclareMathOperator{\const}{const}

\DeclareMathOperator{\diag}{diag}
\DeclareMathOperator{\we}{we}

\DeclareMathOperator{\Cat}{Cat}

\DeclareMathOperator{\Ex}{Ex}

\DeclareMathOperator{\Ob}{Ob}

\DeclareMathOperator{\id}{id}

\DeclareMathOperator{\Ho}{Ho}

\DeclareMathOperator{\RelCat}{RelCat}

\DeclareMathOperator{\Hom}{Hom}

\DeclareMathOperator{\map}{map}
\DeclareMathOperator{\holim}{holim}

\DeclareMathOperator{\sCat}{sCat}

\DeclareMathOperator{\sSet}{sSet}
\DeclareMathOperator{\ssSet}{ssSet}

\DeclareMathOperator{\Map}{Map}

\newcommand{\un}{\underline{n}}
\newcommand{\setms}{\,\setminus\,}
\newcommand{\unn}{\un \setminus \{k\}}

\begin{document}
\title{Fibration Categories are Fibrant Relative Categories}
\author{Lennart Meier}
\begin{abstract}
A relative category is a category with a chosen class of weak equivalences. Barwick and Kan produced a model structure on the category of all relative categories, which is Quillen equivalent to the Joyal model structure on simplicial sets and the Rezk model structure on simplicial spaces. We will prove that the underlying relative category of a model category or even a fibration category is fibrant in the Barwick--Kan model structure. 
\end{abstract}
\maketitle

\section{Introduction}
Abstract homotopy theory comes nowadays in a variety of flavors. A traditional and very rich version is Quillen's theory of model categories, first developed in \cite{Qui67}. More recently, various versions of $\infty$-categories, like Joyal's quasi-categories and Rezk's complete Segal spaces, came into fashion. We will concentrate in this article on maybe the most naive flavor: relative categories. 

A relative category is a category with a chosen class of morphisms, called \emph{weak equivalences}, closed under composition and containing all identities. Despite the apparent simplicity of the definition, only recently Barwick and Kan developed in \cite{B-K12R} a satisfactory homotopy theory of relative categories by exhibiting a model structure on the category $\RelCat$ of (small) relative categories. This model category is Quillen equivalent to the Joyal model structure on simplicial sets and the Rezk model structure on simplicial spaces. 

More precisely, Barwick and Kan consider functors
\[N, N_\xi\colon \RelCat \to \ssSet\]
into simplicial spaces, where $N$ is the Rezk classifying diagram and $N_\xi$ is a variant of it, involving double-subdivision. They lift then the Rezk model structure from $\ssSet$ to $\RelCat$ along $N_\xi$. This is analogous, though technically more demanding, to the Thomason model structure on the category $\Cat$ of (small) categories that is lifted from the standard model structure on $\sSet$ along $\Ex^2\mathrm{Nerve}$. 

Both in the Joyal and in the Rezk model structure the fibrant objects deserve special attention: These are called quasi-categories and complete Segal spaces, respectively. An equally good understanding of the fibrant objects in the Barwick--Kan model structure on $\RelCat$ remains elusive to this day. We will prove, however, a \emph{sufficient} criterion for fibrancy. 

\begin{mthm*}
The underlying relative category of a fibration category $\MM$ is fibrant in the Barwick--Kan model structure. 
\end{mthm*}
More generally, every \emph{homotopically full subcategory} of a fibration category is fibrant as a relative category. Recall here that a fibration category is a generalization of a model category, having just fibrations and weak equivalences and no cofibrations. We will use the term essentially in the sense of \emph{cat\'egories d\'erivables \`a gauche} in \cite{Cis10}. A homotopically full subcategory is a full subcategory closed under the relation of weak equivalence.

In \cite{B-K13}, Barwick and Kan show that every relative category is in their model structure weakly equivalent to a homotopically full subcategory of a model category via a Yoneda-type embedding. Therefore, our results imply that they actually construct an explicit fibrant replacement functor in $\RelCat$. Our main result also allows a rather simple construction of the quasi-category associated to a model category (see Remark \ref{quasi}).

Our main result is equivalent to the statement that $N_\xi \MM$ is a complete Segal space for $\MM$ a fibration category. In a sequence of papers \cite{Rez01}, \cite{Ber09} and \cite{B-K13}, it was shown that \emph{a Reedy fibrant replacement of} $N\MM$ is a complete Segal space if $\MM$ is a (partial) model category. This was generalized in \cite{LMG14} and using this, Low showed in \cite{Low15} that the analogous statement is also true for a fibration category $\MM$. As Barwick and Kan showed in \cite{B-K12R} that there is a Reedy weak equivalence $N\MM \to N_\xi\MM$ for any relative category $\MM$, it remained to show that $N_\xi\MM$ is Reedy fibrant. This is the contribution of the present paper. 

\begin{thm*}
If $\MM$ is a fibration category, then $N_\xi\MM$ is Reedy fibrant. 
\end{thm*}

 Our proof uses ideas from \cite{M-O14}, where we show that the category of weak equivalences of a partial model category is fibrant in the Thomason model structure. 
 
 Note that there is a conjugate of the Barwick--Kan model structure, using a conjugate definition of double-subdivision. The underlying relative categories of \emph{cofibration} categories are fibrant object in this conjugate model structure. \\

We give a short overview of the structure of the article. In Section \ref{BKSection}, we will recall notation and concepts from the theory of (relative) categories. In Section \ref{FibSection} we will treat fibration categories and homotopy limits in them. In Section \ref{MCFSection}, we will give the main steps of our proofs of the two theorems above. In Section \ref{GFSection}, we will provide a proof for the fibrancy criterion used in Section \ref{MCFSection}. In Section \ref{CLSection}, we will give some leftover proofs about the contractibility of certain subsets of simplices. 

\subsection*{Acknowledgments}
This note grew out of collaboration with Viktoriya Ozornova. I thank her for many helpful discussions, for reading earlier versions of this material and for the resulting suggestions that substantially improved exposition and content of this paper. I also thank Zhen Lin Low for a helpful email exchange. 

\section{Homotopy Theory of (Relative) Categories}\label{BKSection}
In this section, we will recall the definition of the Thomason model structure on the category of small categories and of the Barwick--Kan model structure on the category of small relative categories.  \\

Thomason constructed a model structure on the category of small categories $\Cat$:

\begin{thm}\cite{ThomCat} There is a model structure on $\Cat$, where a map $f$ is a weak equivalence/fibration if and only if $\Ex^2\mathrm{Nerve}(f)$ is a weak equivalence/fibration. Equivalently, $f$ is a weak equivalence if and only if $\mathrm{Nerve}(f)$ is a weak equivalence. \end{thm}

Here, $\Ex$ denotes the right adjoint of the subdivision functor $\Sd\colon \sSet \to \sSet$. 

The functor $\mathrm{Nerve}\colon \Cat \to \sSet$ has a left adjoint $c\colon \sSet \to \Cat$, called the \textit{fundamental category functor}. For example, $c\Delta[n] = \un$, the category of $n$ composable morphisms. This defines a Quillen equivalence
\[\xymatrix{ \sSet \ar@/^/[r]^{c\Sd^2} & \Cat \ar@/^/[l]^{\Ex^2\mathrm{Nerve}} }\]
to the Kan model structure on simplicial sets.\\

Barwick and Kan construct an analogous model structure on the category of small relative categories. 

\begin{defi}A \textit{relative category} $\MM$ is a category $\MM$ together with a subcategory $\we \MM$ containing all objects of $\MM$. The morphisms in $\we \MM$ are usually called \textit{weak equivalences}. A \textit{relative functor} between relative categories $\MM$ and $\MM'$ is a functor $F\colon \MM \to \MM'$ with $F(\we \MM)\subset \we \MM'$. We denote the category of (small) relative categories with relative functors between them by $\RelCat$. 
\end{defi}
\begin{remark}
 As we want later to view model categories as objects in $\RelCat$, the usual size issues come up. Two possible solutions are sketched in the introduction of \cite{M-O14} and a more extensive treatment can be found in \cite{ShuCat}. We will ignore these issues in the rest of this article. 
\end{remark}

Given a relative category $(\MM, \we\MM)$, we denote by $\Ho(\MM)$ its \emph{homotopy category}, i.e.\ the localization of $\MM$ at $\we \MM$.

Given a category $\CC$, we denote by $\widehat{\CC}$ its \textit{maximal} relative structure, where every morphism is a weak equivalence, and by $\check{\CC}$ its \textit{minimal} relative structure, where only identities are weak equivalences.\\

We want to define functors $N$ and $N_\xi$ from $\RelCat$ to the category of simplicial spaces $\ssSet$, where we mean by a simplicial space a bisimplicial set. To this purpose, we first have to talk about subdivision of relative posets, i.e.\ posets with the structure of a relative category. In the following, let $\PP$ be a relative poset. 

\begin{defi}[Terminal and initial subdivisions] The \emph{terminal (resp.\ initial) subdivision} of $\PP$ is the relative poset $\xi_t\PP$ (resp.\ $\xi_i\PP$) which has
\begin{enumerate}
 \item  as objects the relative functors $\check{\un} \to \PP$ that are monomorphisms, for $n\geq 0$,
\item as maps $(x_1\colon \check{\un}_1 \to \PP) \to (x_2\colon \check{\un}_2 \to \PP)$ the commutative diagrams of the form
\[\xymatrix{\check{\un}_1\ar[rr]\ar[dr]_{x_1} && \check{\un}_2 \ar[dl]^{x_2}\\
 &\PP& }
\]
for the terminal subdivision and the commutative diagrams of the form 
\[\xymatrix{\check{\un}_2\ar[rr]\ar[dr]_{x_2} && \check{\un}_1 \ar[dl]^{x_1}\\
 &\PP& }
\]
for the initial subdivision.
\item as weak equivalences those of the above diagrams for which the induced
map $x_1(n_1) \to x_2(n_2)$ (resp. $x_2(0) \to x_1(0)$) is a weak equivalence in $\PP$.
\end{enumerate}
The \emph{double-subdivision} $\xi\PP$ is defined as $\xi_t\xi_i\PP$.
\end{defi}

In other words: The subdivision has as objects ascending chains in $\PP$ and the terminal and initial versions correspond to two ways these can be partially ordered. For the terminal subdivision, the last-vertex map
\[\xi_t\PP \to \PP,\quad (x\colon \check{\un} \to \PP) \mapsto x(n)\]
detects the weak equivalences. For the initial subdivision, the initial-vertex map
\[\xi_i\PP \to \PP,\quad (x\colon \check{\un} \to \PP) \mapsto x(0)\]
detects the weak equivalences. Composing last- and initial-vertex map defines a natural transformation $\xi \to \id$.

\begin{example}\label{undescription}Let $\un$ be equipped with an arbitrary relative structure. An object of $\xi(\un)$ can be identified with an ascending non-empty chain of non-empty subsets of $\{0,\dots ,n\}$. If we can build a chain
\[A_\bullet =( A_0\subsetneq \cdots \subsetneq A_m)\]
from a chain
\[B_\bullet = (B_0\subsetneq \cdots \subsetneq B_l)\]
 by adding subsets, then $B_\bullet\leq A_\bullet$. The corresponding morphism is a weak equivalence if $\phi(A_\bullet)\simeq \phi(B_\bullet)$ in $\un$, where $\phi$ is the functor $\xi(\un) \to \un$ sending a chain $A_\bullet$ to the smallest element of $A_0$. 
 
 Note here that $A_0$ is the largest element in the chain $A_\bullet$ in $\xi_i\un$ -- but our numbering system seems more natural to the author than the opposite one. 
 
 Note furthermore that $\mathrm{Nerve}\,\xi(\un)$ is isomorphic to the double barycentric subdivision $\Sd^2\Delta[n]$. It follows that the underlying category of $\xi(\un)$ is isomorphic to $c\Sd^2\Delta[n]$ as $c\,\mathrm{Nerve} \cong \id_{\Cat}$. 
 \end{example}

Next, we define the classifying diagram of a relative category, an analogue of the nerve functor.

\begin{defi}For a relative category $\MM$, we define its \textit{classifying diagram} to be the simplicial space $N\MM$ with $(N\MM)_{pq} = \RelCat(\underline{\check{p}}\times\underline{\hat{q}}, \MM)$. Likewise, we define $N_\xi\MM$ to be the simplicial space with $(N_\xi\MM)_{pq} = \RelCat(\xi(\underline{\check{p}}\times\underline{\hat{q}}), \MM)$.\end{defi}
The natural transformation $\xi \to \id$ induces a natural weak equivalence $N \to N_\xi$, as shown in \cite{B-K12R}. 

Rezk defines in \cite{Rez01} a model structure on the category of simplicial spaces $\ssSet$, where the fibrant objects are the complete Segal spaces. He constructs it as a localization of the usual Reedy model structure. Barwick and Kan lift their model structure on $\RelCat$ from the Rezk model structure on $\ssSet$. 

\begin{thm}\cite{B-K12R} There is a model structure on $\RelCat$, where a map $f$ is a weak equivalence/fibration if and only if $N_\xi f$ is a weak equivalence/fibration in the Rezk model structure on $\ssSet$.\end{thm}
Note here that $N_\xi f$ is a weak equivalence if and only if $Nf$ is one, but there is no analogous statement for fibrations. Barwick and Kan also show in \cite{B-K12A} that $f\colon \MM \to \NN$ is a weak equivalence in $\RelCat$ if and only if it induces a Dwyer--Kan equivalence of the hammock localizations
\[L^H\MM \to L^H\NN,\]
 i.e.\ an equivalence of homotopy categories and weak equivalences of mapping spaces.

The functors 
\[N, N_\xi\colon \RelCat \to \ssSet\]
have left adjoints $K$ and $K_\xi$, respectively. These are the unique colimit-preserving functors with $K(\Delta[p,q]) = \underline{\check{p}}\times\underline{\hat{q}}$ and $K_\xi(\Delta[p,q]) = \xi(\underline{\check{p}}\times\underline{\hat{q}})$, respectively. Here, $\Delta[p,q]$ is the bisimplicial set whose $m$-$n$-simplices $\Delta[p,q]_{mn}$ are the set of maps $(m,n) \to (p,q)$ in $\Delta \times \Delta$. We have diagrams
\[\xymatrix{\ssSet \ar[rr]^{K}\ar[d]^{\mathrm{diag}} && \RelCat \ar[d]^{u} \\
\sSet \ar[rr]^{c} && \Cat
}\] 
and
\[\xymatrix{\ssSet \ar[rr]^{K_\xi}\ar[d]^{\mathrm{diag}} && \RelCat \ar[d]^{u} \\
\sSet \ar[rr]^{c\Sd^2} && \Cat
}\] 
that are commutative up to natural isomorphism, where $u$ denotes the forgetful functor and $\diag(X_{\bullet \bullet})_m = X_{mm}$. Indeed, we have 
\begin{align*}uK\Delta[p,q] & \cong \underline{p}\times\underline{q} \cong c\,\mathrm{Nerve}(\underline{p}\times\underline{q}) \cong c(\Delta[p]\times \Delta[q]) \cong c\diag\Delta[p,q],\\
uK_\xi\Delta[p,q] &\cong u\xi(\underline{\check{p}}\times\underline{\hat{q}}) \cong c\Sd^2(\Delta[p]\times\Delta[q]) \cong c\Sd^2\diag\Delta[p,q] \end{align*}
and $u$ and $\diag$ are both colimit-preserving as $u$ is left adjoint to the functor $\CC \mapsto \widehat{\CC}$.  

\begin{remark}\label{quasi}
The functor $N_\xi\colon \RelCat \to \ssSet$ is actually a right Quillen equivalence from the Barwick--Kan model structure to the Rezk model structure, as shown in \cite[Theorem 6.1]{B-K12R}. There is a further Quillen equivalence
\[\xymatrix{ \sSet \ar@/^/[r]^{p_1^*} & \ssSet \ar@/^/[l]^{i_1^*} }\]
between the Joyal model structure and the Rezk model structure, as shown in \cite{J-T07}. The functor $i_1^*$ assigns to a bisimplicial set $X_{\bullet\bullet}$ its zeroth row $X_{\bullet 0}$. In particular, we have
\[(i_1^*N_\xi \CC)_p = \RelCat(\xi(\underline{\check{p}}), \CC).\]

If $\CC$ is fibrant, $i_1^*N_\xi\CC$ is fibrant in the Joyal model structure, i.e.\ a quasi-category (also known as an $\infty$-category in \cite{Lur09}). As our main theorem states that the underlying relative category of every fibration category is fibrant, this gives a model for the quasi-category associated with a fibration category. 

As explained in \cite{93139}, results by To\"en and Barwick--Kan imply that this is equivalent to other quasi-categories associated with $\MM$, in particular the quasi-category $N_c(L^H\MM)^f$, where $L^H$ is the hammock localization, $f$ denotes a fibrant replacement in the Bergner model structure on simplicial categories and $N_c\colon \sCat \to \sSet$ denotes the coherent nerve. 
\end{remark}

\section{Fibration Categories and Homotopy Limits}\label{FibSection}
Relative categories without extra structure are often hard to work with. Therefore, several mathematicians introduced more structured versions like model categories or fibration categories. We will work with the following definition of a fibration category: 
\begin{defi}
A \textit{fibration category} is a relative category $(\MM, \we \MM)$ together with a subcategory $\FF\subseteq \MM$ of \textit{fibrations}, fulfilling the following axioms:
\begin{itemize}
\item[(F1)] $\MM$ has a terminal object $\ast$. We call an object $x\in \MM$ \emph{fibrant} if $x\to \ast$ is a fibration. We assume $\ast$ to be fibrant. 
\item[(F2)] All isomorphisms are weak equivalences and all isomorphisms with fibrant codomain are fibrations.
\item[(F3)] Let $f$, $g$ and $h$ be composable morphisms. If $gf$ and $hg$ are weak equivalences, then also $f$, $g$ and $h$.
\item[(F4)] Let $f\colon A\to C$ be a morphism between fibrant objects. If $p\colon B\to C$ is a (trivial) fibration, then the pullback $B\times_C A$ exists and the map $B\times_C A \to A$ is also a (trivial) fibration.  Here, a fibration is called a \emph{trivial} if it is a weak equivalence.
\item[(F5)] Any map $f\colon A \to C$ with $C$ fibrant factors as 
\[A \xrightarrow{s} B \xrightarrow{p} C \]
with $s$ a weak equivalence, $p$ a fibration and $ps = f$. 
\end{itemize}
\end{defi}
This agrees essentially with the notion of a \emph{cat\'egories d\'erivables \`a gauche} in the sense of \cite{Cis10} and with the notion of an (ABC) prefibration category from \cite{RadCofibrations}, only that we ask for a $2$ out of $6$ axiom instead of $2$ out of $3$. The latter source also discusses the relationship of this definition with other notions of fibration categories. In particular, every model category is a fibration category by forgetting the cofibrations. 


Next, we will define homotopy limits of diagrams in fibration categories indexed over an arbitrary \emph{finite inverse category}. We follow the treatments in \cite[Sections 1 and 2]{Cis10}\footnote{Beware that Cisinski uses finite \emph{direct} categories as he consideres presheaves instead of covariant functors. Note also that he calls a Reedy fibration \emph{fibration bord\'ee}.} and \cite[Chapter 9]{RadCofibrations}. 

We fix in the following a finite inverse category $\DD$ and a fibration category $(\MM,\we \MM,\FF)$. Then there is a \emph{Reedy fibration category structure} on the functor category $\MM^\DD$ with weak equivalences defined objectwise and Reedy fibrations as fibrations (\cite[Th\'eor\`eme 1.30]{Cis10}).
\begin{thm}[\cite{Cis10}, Proposition 2.6]
The constant diagram functor 
\[\const_\DD\colon \Ho(\MM) \to \Ho(\MM^\DD)\]
 has a right adjoint $\holim_\DD$.
\end{thm}

This adjoint is constructed as follows: Given a functor $F\colon \DD \to \MM$, take a Reedy fibrant replacement $F\to F'$. The limit of $F'$ exists and we have $\holim_\DD F \cong \lim_\DD F'$. 

Given functors $F\colon \DD \to \MM$ and $i\colon \Ac \to \DD$ (with $\Ac$ finite inverse), we have an induced map 
\[u_{\Ac}^\DD\colon \holim_\DD F \to \holim_{\Ac} i^*F,\]
 adjoint to the map 
 \[\const_{\Ac}\holim_\DD F \to i^*F\]
  that is given by applying $i^*$ to the counit $\const_\DD \holim_\DD F \to F$. The map $u_{\Ac}^\DD$ will also be called the \emph{canonical map}. This is a functorial construction in the sense that $u_{\Ac}^\DD =  u_{\Ac}^\BB\circ u_\BB^\DD$ for a diagram of finite inverse categories $\Ac\to \BB\to \DD$, as can be shown by standard properties of adjoints. 

Next, we want to prove three properties of the homotopy limit construction in fibration categories. We will reduce these statements to already known results in the world of model categories via a Yoneda-type construction. 

Following Cisinski, we have the following proposition:
\begin{prop}\label{PropYoneda}
Let $(\MM,\we \MM,\FF)$ be a (small) fibration category.\footnote{The smallness hypothesis can be ensured for our purposes either by the use of universes or by the following observation: If $F\colon \CC \to \MM$ is a functor from a small category $\CC$, then $F$ factors over a small sub fibration category $\MM'\subset \MM$; the homotopy limit of $F$, if $\CC$ is finite inverse, can then be computed in $\MM'$.} Then there exists a functor $h\colon \MM \to \PP_w(\MM)$ into a model category with functorial factorizations, which has the following properties:
\begin{enumerate}
\item $h$ preserves and reflects weak equivalences;
\item if all objects of $\MM$ are fibrant, then $h$ preserves homotopy limits along arbitrary finite inverse categories. 
\end{enumerate}
\end{prop}
\begin{proof}
We will construct $h$ via a Yoneda-type embedding, following \cite[Section 3]{Cis10b}. Let $\PP(\MM)$ be the category of simplicial presheaves on $\MM$ with the projective model structure. Consider the Yoneda
 embedding $h\colon \MM \to \PP(\MM)$ and define $\PP_w(\MM)$ to be the Bousfield localization of $\PP(\MM)$
 at $h(\we \MM)$. Note that 
 $\PP_w(\MM) = \PP(\MM)$
  as categories, but the model structures are different.

Clearly $h\colon \MM \to \PP_w(\MM)$ preserves weak equivalences. We want to show that it also detects weak equivalences.  
Cisinski observes that the $h(\we \MM)$-local objects in $\PP(\MM)$ are exactly those presheaves $\FF$ such that $\FF(Y) \to \FF(X)$ is a weak equivalence if $X\to Y$ is a weak equivalence in $\MM$. For example, the discrete presheaf $ho_X$, defined by $ho_X(Y) = \Ho(\MM)(Y,X)$, is $h(\we \MM)$-local for every $X\in \MM$. Assume now that for a morphism $f\colon X\to Y$ in $\MM$ the morphism $h(f)\colon h(X)\to h(Y)$ is a weak equivalence in $\PP_w(\MM)$. Observe that $h(X)$ and $h(Y)$ are projectively cofibrant and $ho_Z$ is projectively fibrant for every $Z\in\MM$. Thus,
\[ho(\MM)(Y,Z) \cong \map(h(Y), ho_Z) \to ho(\MM)(X,Z) \cong \map(h(X),ho_Z) \]
is a weak equivalence and hence an isomorphism for every $Z\in\MM$. Thus, $X\to Y$ induces an isomorphism in $\Ho(\MM)$ and is thus a weak equivalence by \cite[Theorem 7.2.7]{RadCofibrations}. 

Assume now that all objects of $\MM$ are fibrant. Then by \cite[Proposition 2.1.2]{RadCofibrations}, $\MM$ is a category of fibrant objects in the sense of Brown. By \cite[Corollaire 3.12]{Cis10b}, $h$ preserves fibrations and acyclic fibrations; furthermore, it preserves all limits. Clearly, $h$ preserves thus Reedy fibrant diagrams and hence preserves all homotopy limits along finite inverse categories. 
\end{proof}

\begin{prop}\label{ContractibleHolim}Let $\MM$ be a fibration category and $\DD$ be a finite inverse category whose nerve is contractible. Let $F\colon \DD \to \we \MM$ be a diagram. 
Then the morphism
\[u_d^{\DD}\colon \holim_{\DD}F \to F(d)\]
is an isomorphism in $\Ho(\MM)$ for every $d\in\DD$.
\end{prop}
\begin{proof}
By a (Reedy) fibrant replacement, we can replace $F$ by a diagram in the subcategory of fibrant objects $\MM_{fib}$ with the same homotopy limit (computed in $\MM_{fib})$. Thus, we can assume that every object of $\MM$ is fibrant. The result follows now from the corresponding result for model categories \cite[Corollary 29.2, Section 31]{C-S02} and Proposition \ref{PropYoneda}.
\end{proof}

\begin{defi}
A functor $i\colon \Ac \to \BB$ between two categories is called \emph{homotopically initial} if $\mathrm{Nerve}\, (i/b)$ is (weakly) contractible for every $b\in B$, where $i/b$ denotes the comma category. 
\end{defi}
\begin{prop}\label{FinalHolim}
Let $\MM$ be a fibration category and $\DD$ be a finite inverse category. Let $i\colon \Ac \to \DD$ be a homotopically initial functor and $F\colon \DD \to \MM$ a diagram. Then the canonical map
\[\holim_\DD F \to \holim_\Ac (Fi)\]
is an isomorphism in $\Ho(\MM)$.
\end{prop}
\begin{proof}
 As before, we can assume that $\MM$ has only fibrant objects. The result follows now from the corresponding result for model categories \cite[31.6]{C-S02} and Proposition \ref{PropYoneda}.
\end{proof}

For the following proposition recall that a full subcategory $\Ac\subset \DD$ is called a \emph{cosieve} if for every $a\in \Ac$ and every morphism $a\to d$ in $\DD$, we already have $d\in \Ac$. 
\begin{prop}\label{DecompositionHolim}
Let $\MM$ be a fibration category and $\DD$ be a finite inverse category. Let $\Ac,\BB\subset \DD$ be inclusions of cosieves. Let $F\colon \DD \to \MM$ be a diagram. Then there is an isomorphism
\[\holim_\DD F \to \holim_\Ac F \times^h_{\holim_{\Ac\cap \BB}F} \holim_\BB F\]
in $\Ho(\MM)$, compatible with the canonical maps to $\holim_\Ac F$ and $\holim_\BB F$. 
\end{prop}
\begin{proof}
As before, we can assume that $\MM$ has only fibrant objects. The result follows now from the corresponding result for model categories \cite{C-S02}[31.5] and Proposition \ref{PropYoneda} as follows: Chacholski and Scherer prove that
\[\holim_{\EE} F \to \holim_\Ac F \times^h_{\holim_{\Ac\cap \BB}F} \holim_\BB F\]
is an equivalence, where $\EE$ is a co-Grothendieck construction, which is in our case given as follows: It has objects
\begin{itemize}
\item $(a,0)$ for $a\in\Ac$,
\item $(b,1)$ for $b\in\BB$, and
\item $(c,01)$ for $c \in \Ac\cap \BB$.
\end{itemize}
The morphisms $(d,i) \to (d', i')$ are morphisms $d\to d'$ in $\DD$ if $i= i'$ or $i=0$ or $1$ and $i' = 01$. 

We will show that the functor
\[G\colon \EE \to \DD,\qquad (d,i) \mapsto d\]
is homotopically initial. The category $G/d$ has 
\[\xymatrix{
(d,i)\ar@{|->}[d] \\
d \ar[r]^= & d 
}\]
as terminal object with
\begin{itemize}
\item $i=0$ if $d\in \Ac$, but $d\notin \BB$,
\item $i=1$ if $d\in \BB$, but $d\notin \Ac$, and
\item $i=01$ if $d\in \Ac\cap \BB$.
\end{itemize} 
In the first two cases, we use that $\Ac\subset \DD$ and $\BB\subset \DD$ are cosieves. 

Thus,
\[u_{\EE}^{\DD}\colon \holim_\DD F \to \holim_\EE F\]
 is an equivalence by Proposition \ref{FinalHolim} and the result follows.
\end{proof}

\begin{remark}
Instead of using \cite{C-S02}, we could also have used the language of quasi-categories. The model category $\PP_w(\MM)$ is actually simplicial by \cite[Theorem 4.46]{Bar10}. By \cite[Theorem 4.2.4.1]{Lur09}, homotopy limits in $\PP_w(\MM)$ and in the coherent nerve $N_c\PP_w(\MM)^\circ$ agree. Thus, the last three propositions follow from the corresponding results in quasi-categories: \cite[Corollary 4.4.4.10]{Lur09}, \cite[Theorem 4.1.3.1 and Proposition 4.1.1.8]{Lur09} and \cite[Corollary 4.2.3.10]{Lur09}.
\end{remark}

\section{Model Categories are Fibrant}\label{MCFSection}
Our main goal in this section is to prove that $N_\xi\MM$ is Reedy fibrant if $\MM$ is a fibration category. This will imply then that every fibration category is fibrant as a relative category in the Barwick--Kan model structure. 

First, we have to introduce the following notation: For a category $\DD$, let $\KK(\DD)$ be the category $\DD \times (0\to 1) \cup_{\DD \times 1} \DD^{\vartriangleleft}$, where $\DD^{\vartriangleleft}$ denotes the category $\DD$ with an additional initial object. Thus, $\KK(\DD)$ consists of two copies of $\DD$, where there is a unique map from the $0$-copy of each object to the $1$-copy of it, and each object in the $1$-copy receives an additional morphism from a ``partial initial object''. We will view $\DD$ as a subcategory of $\KK(\DD)$ via the identification $\DD \cong \DD \times 0$. We will furthermore denote the ``partial initial object'' by $k_\DD \in \KK(\DD)$. 

In \cite[Lemma 4.2]{M-O14}, we showed the following fibrancy criterion for the Thomason model structure:
\begin{prop}\label{OldFibrancy}
  A category $\CC$ is fibrant in the Thomason model structure if and only if it has the right lifting property with respect to all inclusions $c\Sd^2 \Lambda^n[n] \to \KK(c\Sd^2 \Lambda^n[n])$. 
\end{prop}

Our first aim is to show that the category of weak equivalences of a fibration category is fibrant in the Thomason model structure. The following proposition will be key:

\begin{prop}\label{ExtensionProp}Let $\DD$ be an arbitrary finite inverse category and $F\colon \DD \to \MM$ be a functor for $\MM$ a fibration category. Then one can extend $F$ to a functor $G\colon \KK(\DD) \to \MM$ such that $G((d,0) \to (d,1))$ is a weak equivalence for every $d\in\DD$ and $G|_{\DD^{\vartriangleleft}}$ is a homotopy limit diagram.\end{prop}
\begin{proof}
We can find a weak equivalence $F\to F'$ to a Reedy fibrant diagram, corresponding to a functor $\nu\colon D\times \underline{1} \to \MM$. As dicussed in the previous section, limits of Reedy fibrant diagrams exist and are homotopy limits. Let $\widetilde{F'}\colon \DD^{\vartriangleleft} \to \MM$ be a limit cone for $F'$. Then we can glue $G$ from $\nu$ and $\widetilde{F'}$. 
\end{proof}

The following corollary also follows from our later results, but we prefer to give a direct proof. 
\begin{cor}
 The category of weak equivalences of a fibration category is fibrant in the Thomason model structure.
\end{cor}
\begin{proof}
 The category $c\Sd^2\Lambda^k[n]$ is inverse. Indeed, $c\Sd^2\Lambda^k[n]\subset c\Sd^2\Delta[n]$ can be viewed as consisting of chains of subsets of $\underline{n}$ and the length of the chain provides the inverse structure. 
 
Let now $(\MM, \we\MM, \FF)$ be a fibration category and $F\colon c\Sd^2 \Lambda^n[n] \to \we \MM$ be a diagram. By Proposition \ref{ExtensionProp} and Proposition \ref{ContractibleHolim}, we can extend $F$ to a diagram $\KK(c\Sd^2 \Lambda^n[n]) \to \we \MM$. Proposition \ref{OldFibrancy} implies the statement. 
\end{proof}

\begin{remark}The proof in \cite{M-O14} of the fibrancy of partial model categories was considerably harder as no analogue of Reedy fibrant replacement for functors indexed by inverse categories exists for general partial model categories.
\end{remark}

Showing fibrancy in the Barwick--Kan model structure is more complicated than in the Thomason model structure. Before we formulate a fibrancy criterion, we have to discuss certain preliminaries. 

We can identify $\KK(c\Sd^2\partial \Delta[n])$ with $\xi \un$ as categories as follows: Objects in $c\Sd^2\partial \Delta[n]$ can be identified with ascending non-empty chains 
\[A_\bullet = (A_0\subsetneq \cdots \subsetneq A_m)\]
of non-empty subsets of $\un$ such that $A_m \neq \un$.\footnote{Here and in the following, we abuse notation by using $\un$ both for the category of $n$ composable morphisms and for $\{0,1,\dots, n\}$, its set of objects.} For such a chain, we identify $(A_\bullet, 1)$ in $\KK(c\Sd^2\partial \Delta[n])$ with $A_0\subsetneq \cdots \subsetneq A_m \subsetneq \un$ and $k_{c\Sd^2\partial \Delta[n]}$ with the chain just consisting of $\un$ in $\xi(\un)$. We refer to \cite[Remark 4.1]{M-O14} for a picture of this identification. If we choose a relative structure on $\un$, the relative structure of $\xi \un$ defines thus a relative structure on  $\KK(c\Sd^2\partial \Delta[n])$ and thus also on $\KK(c\Sd^2\Lambda^k[n])$ for every $0\leq k\leq n$. 

We are now ready to formulate the following fibrancy criterium that will be proved in slightly stronger form as Proposition \ref{FibrancyCriteriumProp}.

\begin{prop}\label{FibCrit}
Let $\MM$ be relative category. Assume that $\MM$ has the right lifting property with respect to all
\[c\Sd^2\Lambda^k[n] \to \KK(c\Sd^2\Lambda^k[n])\]
for $n\geq 1$ and $0\leq k \leq n$, where the relative structure on $\KK(c\Sd^2\Lambda^k[n])$ is induced by an arbitrary relative structure on $\underline{n}$ such that $(n-1)\to n$ is a weak equivalence if $k=n$. Then $N_\xi\MM$ is Reedy fibrant.
\end{prop}

Let now and in the following $n\geq 1$ and $0\leq k\leq n$ be fixed numbers. Equip $\un$ with an arbitrary relative structure such that $(n-1)\to n$ is a weak equivalence if $k=n$. Set for the rest of the section $\DD = c\Sd^2\Lambda^k[n]$, with relative structure induced by that on $\un$. We now want to describe the weak equivalences in $\DD$ more concretely: 

 The functor 
\[\pi = \phi|_\DD\colon \DD \subset \xi\un \xrightarrow{\phi} \un\]
described in Example \ref{undescription} detects and preserves weak equivalences. This implies the following description of weak equivalences: All morphisms $(A_\bullet, 0) \to (A_\bullet, 1)$ are weak equivalences. A morphism $(A_\bullet, i)\to (B_\bullet, i)$ for $i=0$ or $1$ is a weak equivalence if and only if $\pi(A_\bullet) \simeq \pi(B_\bullet)$ in $\un$. Furthermore, $k_{\DD} \to A_\bullet \times 1$ is a weak equivalence if and only if $\pi(A_\bullet) \simeq 0$ in $\un$. \\

Let now and in the following $\MM$ be a fibration category and $F\colon \DD \to \MM$ be a relative functor. To apply Proposition \ref{FibCrit}, we need to show that the functor $G\colon \KK(\DD) \to \MM$ constructed in Proposition \ref{ExtensionProp} is actually a relative functor. This is clear for $n=1$, so we will assume that $n\geq 2$ in the following. Then the following proposition implies exactly that. 

\begin{prop}\label{HolimProp}Let $F\colon \DD \to \MM$ be a relative functor. Then $\holim_\DD F \to F(0)$ is a weak equivalence. Here, we identify $0$ with the object of $\DD$ corresponding to the chain of subsets of $\un$ just consisting of $\{0\}$, i.e.\ with the $0$-corner.\end{prop}
The basic intuition is that after collapsing all weak equivalences to identities, $c\Sd^2\Lambda^k[n]$ becomes a quotient of $\un$ with $0$ as initial object. Of course, more care has to be taken for an actual proof. We will proceed inductively over the $\pi^{-1}(\underline{i})$ for $i\leq n$ and need for that a few intermediate results.

\begin{lemma}\label{PieLimit}Let $i, j$ be integers with $0\leq j\leq i \leq n$. Then there is a homotopy pullback diagram
 \[\xymatrix{
 \holim_{\pi^{-1}(\underline{i})}F \ar[r]\ar[d] & \holim_{\pi^{-1}(\underline{j})}F \ar[d] \\  \holim_{\pi^{-1}(\underline{i}\setms\underline{j})}F \ar[r] & \holim_{\pi^{-1}(\underline{i}\setms\underline{j})\cap V_+\pi^{-1}(\underline{j})}F 
 }\]
 where the horizontal maps and the left vertical map are the canonical maps.
Here, $V_+\CC \subset \DD$ denotes for a subcategory $\CC \subset \DD$ the full subcategory of all $d\in\DD$ such that a morphism $c\to d$ in $\DD$ with $c\in\CC$ exists. 
\end{lemma}
\begin{proof}
 We should first explain the right vertical map inside the homotopy pullback diagram. We claim that the canonical map
 \[\holim_{\pi^{-1}(\underline{i}) \cap V_+\pi^{-1}(\underline{j})}F \to \holim_{\pi^{-1}(\underline{j})}F\]
is an equivalence. Indeed, we can describe the two relevant categories as:
\begin{align*}
\pi^{-1}(\underline{j}) &= \{w_0\subsetneq \cdots \subsetneq w_m\,|\,c\in w_0\text{ for some }c\leq j \}\\
\pi^{-1}(\underline{i}) \cap V_+\pi^{-1}(\underline{j}) &= \{w_0\subsetneq \cdots \subsetneq w_m\,|\, a\in w_0\text{ for some }a\leq i\text{ and }b\in w_m \text{ for some }b\leq j \} 
 \end{align*}
For an arbitrary
\[d =(w_0\subsetneq \cdots \subsetneq w_m)\in \pi^{-1}(\underline{i}) \cap V_+\pi^{-1}(\underline{j}),\] the category $\pi^{-1}(\underline{j})/d$ has a terminal object: Just delete every $w_r$ that does not contain some $c\leq j$. Thus, $\pi^{-1}(\underline{j}) \to \pi^{-1}(\underline{i}) \cap V_+\pi^{-1}(\underline{j})$ is cofinal and the homotopy limits agree by Proposition \ref{FinalHolim}.

Next we observe that $\pi^{-1}(\underline{i}) \cap V_+\pi^{-1}(\underline{j}) \subset \pi^{-1}(\underline{i})$ and $\pi^{-1}(\underline{i}\setms\underline{j}) \to \pi^{-1}(\underline{i})$ are cosieves. Thus, the result follows by Proposition \ref{DecompositionHolim}.
\end{proof}

The following lemma will be proven in Section \ref{CLSection}.

\begin{lemma}\label{Contractible} The nerves of the categories
\begin{itemize}
\item $\pi^{-1}(\EE)$ for every non-empty connected subcategory $\EE\subseteq \un$ that is not $\un \setminus \{k\}$,
\item $\pi^{-1}(i) \cap V_+\pi^{-1}(\underline{i-1})$ for $i\geq 1$ if $k<n$ or $i< n-1$, and 
\item $\pi^{-1}((n-1)\to n) \cap V_+\pi^{-1}(\underline{n-2})$ for $k=n\geq 2$.
\end{itemize}  
are (weakly) contractible. 
\end{lemma}

The next lemma follows now easily. 
\begin{lemma}\label{lastlemma}The maps 
\begin{itemize}
\item[(i)] $\holim_{\pi^{-1}(i)} F \to F(i)$, 
\item[(ii)] $\holim_{\pi^{-1}(i)} F \to \holim_{\pi^{-1}(i)\cap V_+\pi^{-1}(\underline{i-1})}F$ for $i\geq 1$ if $k<n$ or $i< n-1$, and
\item[(iii)]  $\holim_{\pi^{-1}((n-1) \to n)} F \to \holim_{\pi^{-1}((n-1)\to n)\cap V_+\pi^{-1}(\underline{n-2})}F$ for $k=n$ 
\end{itemize} 
are weak equivalences.\end{lemma}
\begin{proof}
As $\mathrm{Nerve}\, \pi^{-1}(i) \simeq \ast$ by Lemma \ref{Contractible}, Proposition \ref{ContractibleHolim} implies part (i). The same argument implies that source and target in (ii) are equivalent to $F(\{i\} \subset \{i-1,i\})$ and so (ii) follows from the 2-out-of-3 principle. 

Recall that for $k=n$ the map $(n-1)\to n$ is a weak equivalence in $\un$. Thus, by the same argument both source and target in (iii) are equivalent to $F(\{n-1\}\subset \{n-2,n-1\})$ and (iii) follows again from the 2-out-of-3 principle. 
 \end{proof}
 
 We are now ready for the proof of Proposition \ref{HolimProp}. 
 \begin{proof}[Proof of Proposition \ref{HolimProp}]
 Assume first that $k<n$. By Lemma \ref{lastlemma}, 
\[\holim_{\pi^{-1}(i)} F \to \holim_{\pi^{-1}(i)\cap V_+\pi^{-1}(\underline{i-1})}F\]
 is a weak equivalence for every $i\geq 1$. Thus, every homotopy pullback along this map is a weak equivalence, in particular, using Lemma \ref{PieLimit}, the map
\[\holim_{\pi^{-1}(\underline{i})}F \to \holim_{\pi^{-1}(\underline{i-1})}F.\]
By Lemma \ref{lastlemma}, it follows that the map
\[\holim_{\pi^{-1}(\underline{n})}F \to \cdots \to \holim_{\pi^{-1}(\underline{0})}F \to F(0)\]
is a weak equivalence for every $i\geq 1$. This shows Proposition \ref{HolimProp} in the case $k<n$. 

The same arguments show that $\holim_{\pi^{-1}(\underline{n-2})}F \to F(0)$ and the map from 
\[\holim_{\pi^{-1}(\underline{n})}F \simeq \holim_{\pi^{-1}((n-1)\to n)}F\times^h_{\pi^{-1}((n-1)\to n) \cap V_+\pi^{-1}(\underline{n-2})} \holim_{\pi^{-1}(\underline{n-2})}F\]
to  $\holim_{\pi^{-1}(\underline{n-2})}F$ are weak equivalences. Therefore their composition 
\[\holim_{\pi^{-1}(\underline{n})}F \to F(0)\]
 is also a weak equivalence in this case. 
 \end{proof}

As discussed above, Proposition \ref{HolimProp} implies:

\begin{thm}\label{TheoremFibration}
For a fibration category $\MM$, the simplicial space $N_\xi \MM$ is Reedy fibrant.
\end{thm}

We now want to deduce our main theorem from this. Recall to that purpose that a full subcategory $\CC \subset \MM$ is called \textit{homotopically full} if $x\in \CC$ and $x\simeq y$ in $\MM$ already imply $y\in \CC$. The crucial ingredient is the following theorem (known in this form at least to Zhen Lin Low).
\begin{thm}[Low, Mazel-Gee, Cisinski]\label{Thm:LMGC}
Let $\MM$ be a fibration category and $\CC \subseteq \MM$ be a homotopically full subcategory. Then a Reedy fibrant replacement of $N(\CC,\we\CC)$ is a complete Segal space. 
\end{thm}
\begin{proof}
The full subcategory of fibrant objects $\MM^\circ$ of $\MM$ forms a category of fibrant objects in the sense of Brown. Set $\CC^\circ = \MM^\circ \cap \CC$. 

By \cite[Theorem A.5]{Low15}, combined with \cite[Lemma 3.7, Lemma 3.11 and Remark 2.9]{Low15}, $\CC^\circ$ has a homotopical three-arrow calculus. By the main result of \cite{LMG14}, this implies that a Reedy fibrant replacement of $N(\CC^\circ, \we\CC^\circ)$ is a complete Segal space. 

A theorem of Cisinski (\cite[Theorem A.3]{Low15}) implies that the inclusions 
\[\mathrm{Nerve }\we ((\MM^\circ)^{\underline{p}}) \to \mathrm{Nerve }\we (\MM^{\underline{p}})\]
 are weak equivalences for all $p\geq 0$. As $\CC\subseteq \MM$ is homotopically full, this implies that also the inclusions $\mathrm{Nerve }\we (\CC^\circ)^{\underline{p}} \to \mathrm{Nerve }\we \CC^{\underline{p}}$ are weak equivalences. Therefore, $N(\CC, \we\CC)$ and $N(\CC^\circ, \we\CC^\circ)$ are Reedy equivalent. Thus, a Reedy fibrant replacement of $N(\CC, \we\CC)$ is a complete Segal space.  
\end{proof}

\begin{thm}Every fibration category is fibrant in the Barwick--Kan model structure.
\end{thm}
\begin{proof}Let $\MM$ be a fibration category. The natural map $N\MM \to N_\xi \MM$ is a Reedy equivalence as shown in \cite{B-K12R}. By the Theorem \ref{TheoremFibration}, it follows that $N_\xi\MM$ is a Reedy fibrant replacement of $N\MM$ and therefore fibrant in the Rezk model structure by the last theorem. As fibrations in $\RelCat$ are defined via $N_\xi$, it follows that $\MM$ is fibrant in $\RelCat$. 
\end{proof}

A slight variant of the proof gives actually the following stronger theorem:

\begin{thm}Every homotopically full subcategory of a fibration category is fibrant in the Barwick--Kan model structure.\end{thm}
\begin{proof}Let $\CC$ be a homotopically full subcategory of a fibration category $\MM$. By \ref{Thm:LMGC}, we only have to show that $N_\xi\CC$ is Reedy fibrant. 

For Theorem \ref{TheoremFibration}, we have checked the fibrancy criterion \ref{FibCrit}. More precisely, we have shown that  $\MM$ has the right lifting property with respect to all
\[c\Sd^2\Lambda^k[n] \to \KK(c\Sd^2\Lambda^k[n]),\]
where the relative structure on $\KK(c\Sd^2\Lambda^k[n])$ is induced by an arbitrary relative structure on $\underline{n}$ such that $(n-1)\to n$ is a weak equivalence if $k=n$.

We now want to check \ref{FibCrit} also for $\CC$. Choose a relative structure on $\un$ as above. Let $F\colon c\Sd^2\Lambda^k[n] \to \CC$ be a functor and $G\colon \KK(c\Sd^2\Lambda^k[n]) \to \MM$ be an extension. Let $x\in c\Sd^2\Lambda^k[n]$ be arbitrary. As $G(x,1) \simeq G(x,0) = F(x)$ and $G(k_{c\Sd^2\Lambda^k[n]}) \simeq G(\{0\}, 1)$, the functor $G$ actually factors over $\CC$. \end{proof}

We now want to indicate, what happens if one considers cofibration categories instead of fibration categories. Define $\overline{\xi} = \xi_i\xi_t$. Then there is a functor 
\[N_{\overline{\xi}}\colon \RelCat \to \ssSet,\]
where $N_{\overline{\xi}}(\CC) = \RelCat(\overline{\xi}(\underline{\check{p}}\times\underline{\hat{q}}),\CC)$ for a relative category $\CC$. Barwick and Kan define in \cite{B-K12R} a \emph{conjugate model structure} on $\RelCat$, where a morphism $f$ is a fibration or weak equivalences if and only if $N_{\overline{\xi}}(f)$ is in the Rezk model structure.

For our purposes, a cofibration category consists of relative category $(\MM, \we\MM)$ together with a subcategory $\CC\subseteq \MM$ such that $(\MM^{op}, \we\MM^{op},\CC^{op})$ is a fibration category. By \cite[Theorem 6.4]{B-K12R} a relative category $\MM$ is fibrant in the conjugate model structure if and only $\MM^{op}$ is fibrant in the usual Barwick--Kan model structure. We obtain:
\begin{cor}
Every cofibration category is fibrant in the conjugate Barwick--Kan model structure.
\end{cor}

\section{General Fibrancy Criteria}\label{GFSection}
In this section, we will give criteria for the Reedy fibrancy of $N_\xi\MM$, where $\MM$ will be throughout an arbitrary relative category. This will culminate in Proposition \ref{FibrancyCriteriumProp}, which is the relevant criterion for Section \ref{MCFSection}.

To use the notion of Reedy fibrancy, we have to view simplicial spaces now no longer as bisimplicial sets, but as simplicial objects in simplicial sets instead; more precisely, we view a bisimplicial set $K_{\bullet \bullet}$ now as a simplicial object $K_\bullet$ with $K_m = K_{m\bullet}$. 
  
There are two ways to view a simplicial set as a simplicial space, a horizontal and a vertical one. For a simplicial set $K$, let $K^h$ be the simplicial space with $(K^h)_m = K$ for all $m$. Furthermore, let $K^v$ be the simplicial space with $(K^v)_m = K_m$, where we view a set as a discrete simplicial set. Note that $\Delta[n]^h\times \Delta[m]^v \cong \Delta[m,n]$; in particular, \[\Hom_{\ssSet}(\Delta[n]^h\times \Delta[m]^v, X_{\bullet\bullet}) = X_{mn}.\]

\begin{lemma}\label{LiftingLemma1}The simplicial space $N_\xi\MM$ is Reedy fibrant if and only if $\MM$ has the right lifting property in $\RelCat$ with respect to 
\[c\Sd^2(\Lambda^k[n]\times \Delta[m]\cup_{\Lambda^k[n]\times \partial \Delta[m]} \Delta[n]\times \partial \Delta[m]) \to c\Sd^2(\Delta[n]\times \Delta [m]) ,\]
for all $n \geq 1$, $m\geq 0$ and $0\leq k \leq n$, where the target inherits its relative structure by the identification with $\xi(\widehat{\underline{n}}\times \check{\underline{m}})$ and the relative structure on the source is induced by that on the target.\end{lemma}
\begin{proof}
 The simplicial space $N_\xi \MM$ is Reedy fibrant iff a lift exists in all diagrams
\[\xymatrix{\Lambda^k[n]\ar[d]\ar[r] & (N_\xi\MM)_m\ar[d] \\
 \Delta[n] \ar[r]\ar@{-->}[ur] & M_m(N_\xi\MM) }
\]
for all $n\geq 1$, $m\geq 0$ and $0\leq k\leq n$. Here, $M_m$ denotes the matching object.

We have 
\[(N_\xi \MM)_m = \Map(\Delta[m]^v, N_\xi\MM)\]
 and 
\[M_mN_\xi\MM = \Map(\partial \Delta[m]^v, N_\xi\MM)\]
 with the map between them induced by the inclusion of the source (see \cite[Chapter IV.3, pp.218-219]{G-J99}). In general, for simplicial spaces $X$ and $Y$, we take as mapping space $\Map(X,Y)$ the simplicial set with $l$-simplices $\Hom(X\times\Delta[l]^h, Y)$. 

By adjunction, a lift in the diagram above is now equivalent to a lift in the diagram
\[\xymatrix{\Lambda^k[n]^h\times \Delta[m]^v \cup_{\Lambda^k[n]^h\times \partial \Delta[m]^v} \Delta[n]^h\times \partial \Delta[m]^v\ar[d]\ar[r] & N_\xi\MM \\
\Delta[n]^h\times \Delta[m]^v\ar@{-->}[ur] 
}\]
and by another adjunction equivalent to a lift in the diagram
\[\xymatrix{K_\xi(\Lambda^k[n]^h\times \Delta[m]^v \cup_{\Lambda^k[n]^h\times \partial \Delta[m]^v} \Delta[n]^h\times \partial \Delta[m]^v\ar[d]\ar[r]) &\MM \\
K_\xi(\Delta[n]^h\times \Delta[m]^v)\ar@{-->}[ur] 
}\]
As explained in Section \ref{BKSection}, this lifting problem is isomorphic to 
\[\xymatrix{c\Sd^2(\Lambda^k[n]\times \Delta[m] \cup_{\Lambda^k[n]\times \partial \Delta[m]} \Delta[n]\times \partial \Delta[m]\ar[d]\ar[r]) &\MM \\
\xi(\widehat{\underline{n}}\times \check{\underline{m}})\ar@{-->}[ur] 
}\]
where the upper left corner inherits a relative structure from $\xi(\widehat{\underline{n}}\times \check{\underline{m}})$. 
\end{proof}

\begin{lemma}\label{LiftingLemma2}The simplicial space $N_\xi\MM$ is Reedy fibrant if $\MM$ has the right lifting property with respect to all
\[c\Sd^2\Lambda^k[n] \to \xi\un,\]
with $0\leq k\leq n$ and $n\geq 1$, where $\un$ carries an arbitrary relative structure satisfying the following conditions:
\begin{itemize}
\item $\un$ has at least one non-identity weak equivalence, 
\item $0\to 1$ in $\un$ is a weak equivalence if $k=0$,
\item  $n-1 \to n$ in $\un$ is a weak equivalence if $k=n$.
\end{itemize}  Here, $c\Sd^2\Lambda^k[n]$ inherits the relative structure from $\xi\underline{n}$. \end{lemma}
\begin{proof}
Denote by $I$ the collection of inclusions $c\Sd^2\Lambda^k[n] \to \xi\underline{n}$ of relative categories as described in the statement of the lemma. We want to show that 
\[c\Sd^2(\Lambda^k[n]\times \Delta[m]\cup_{\Lambda^k[n]\times \partial \Delta[m]} \Delta[n]\times \partial \Delta[m]) \to c\Sd^2(\Delta[n]\times \Delta [m]) = \xi(\widehat{\underline{n}}\times \check{\underline{m}})\]
is in $I$-cell, which means in our case that it can be build by iterative pushouts along maps in $I$. Then Lemma \ref{LiftingLemma1} implies that $\MM$ having the right lifting property with respect to $I$ is sufficient for the Reedy fibrancy of $N_\xi\MM$ . 

Recall that a marked simplicial set is a simplicial set $S$ with a subset $E\subset S_1$ of marked edges, containing all degenerate ones. Call an inclusion of marked simplicial sets with underlying map $\Lambda^k[n] \to \Delta[n]$ an
\begin{itemize}
\item \textit{inner horn} if $0<k<n$,
\item \textit{special left horn} if $k=0$ and $0\to 1$ is marked,
\item \textit{special right horn} if $k=n$ and $(n-1) \to n$ is marked.
\end{itemize}
We denote by $J = J_l$ the collection of inner and special left horns and by $J_r$ the collection of inner and special right horns. 

Nerve and $c$ extend to an adjunction between relative categories and marked simplicial sets, compatible with forgetful functors. We define $\Sd^2$ on marked simplicial sets to be the unique colimit-preserving endofunctor such that $\Sd^2\mathrm{Nerve}\, \CC = \mathrm{Nerve}\, \xi \CC$ for $\CC$ a relative poset. 

As $c$ and $\Sd^2$ are left adjoints and preserve therefore pushouts, it is enough to show that the inclusion
\[\phi = \phi_{k,n,m}: \Lambda^k[n]\times \Delta[m]\cup_{\Lambda^k[n]\times \partial \Delta[m]} \Delta[n]\times \partial \Delta[m] \to \Delta[n]\times \Delta [m]\]
with $\Delta[n]$ maximally and $\Delta[m]$ minimally marked, is in $J$-cell for $k<n$ and in $J_r$-cell for $k>0$. Here, an edge is marked in the product if it is the product of two marked edges. 

We will use the idea of the Box Product Lemma of \cite[Appendix A]{D-S11}. Their proof essentially gives that for $k>0$ the map $\phi$ is in $J_r$-cell. Therefore, we will only do the case $k<n$. Our proof will be dual to that of \cite{D-S11} and we will follow their approach closely. 

Let $Y = \Delta[n]\times\Delta[m]$ and let $Y^0 = \Lambda^k[n]\times \Delta[m]\cup_{\Lambda^k[n]\times \partial \Delta[m]} \Delta[n]\times \partial \Delta[m]$. We will produce a filtration 
\[Y^0 \subset Y^1 \subset \cdots \subset Y^{m+1} = Y\]
 and prove that each $Y^i \to Y^{i+1}$ is in $J$-cell.

Let us establish some notation. An $r$-simplex $y$ in $Y$ is determined by its
vertices, and we can denote it in the form
\[\begin{bmatrix}a_0 & a_1 & \cdots & a_r \\
   u_0 & u_1 & \cdots & u_r 
  \end{bmatrix}
\]
where $0 \leq a_i \leq a_{i+1} \leq n$ and $0 \leq u_i \leq u_{i+1} \leq m$, for $0 \leq i < r$. Faces and
degeneracies are obtained by omitting or repeating columns. 
The simplex $y$
is thus degenerate if and only if two successive columns are identical. 

One checks that the simplex $y$ is an element of $Y^0$ if and only if it satisfies one
of the following two conditions:
\begin{itemize}
 \item[(i)] $\{a_0, a_1, \dots, a_r\}$ equals neither $\{0, 1, \dots, n\}$ nor $\{0, 1, \dots, n\} \setms \{k\}$, or
 \item[(ii)]$\{u_0 , u_1\dots , u_r\} \neq \{0, 1, \dots , m\}$.
\end{itemize}

Let $Y^1$ be the simplicial set generated by the union of $Y^0$ together with all simplices that contain the vertex $\begin{bmatrix}k\\m\end{bmatrix}$,
 and in general let $Y^i$ be the simplicial set generated by the union of $Y^{i-1}$ together with all simplices containing $\begin{bmatrix}k\\m-i+1\end{bmatrix}$. Note that $Y^{m+1} = Y$: Every simplex either contains some $\begin{bmatrix}k\\m-i+1\end{bmatrix}$ or is a face of such a simplex. 

Our goal is to show that each inclusion $Y^i\to Y^{i+1}$ is in $J$-cell, and we
will do this by producing another filtration
\[Y^i = Y^i[n - 1] \subset Y^i[n] \subset \cdots \subset Y^i[n + r] = Y^{i+1}.\]
Notice that every simplex of $Y$ of dimension $n - 1$ or less, containing $\begin{bmatrix}k\\m-i\end{bmatrix}$, lies in $Y^0$
as it satisfies condition (i). For $t > n - 1$ we define $Y^i[t]$ to be generated by the union of $Y^i[t - 1]$
and all nondegenerate simplices of $Y$ that have dimension $t$ and contain $\begin{bmatrix}k\\m-i\end{bmatrix}$. We
claim that $Y^i[t] \to Y^i[t + 1]$ is a cobase change of special left horn
inclusions; justifying this will conclude our proof.

Let $y$ be a nondegenerate simplex of $Y$ of dimension $t + 1 \geq n$ such that $y\in Y^i[t+1]$, but $y \notin Y^i[t]$; in particular, $y$ contains $\begin{bmatrix}k\\m-i\end{bmatrix}$. Then every face of $y$ except possibly for the $\begin{bmatrix}k\\m-i\end{bmatrix}$-face is contained in $Y^i[t]$. We must show that $Y^i[t]$ cannot contain this final
face of $y$, and also that this final face is not the face of another nondegenerate simplex in $Y^i[t+1]$. Given the former, the latter is clear since two different simplices cannot have the same $\begin{bmatrix}k\\m-i\end{bmatrix}$-face.

Now we want to show that the $\begin{bmatrix}k\\m-i\end{bmatrix}$-face $dy$ of $y$ is not in $Y^i[t]$. Write $y$ as 
\[\begin{bmatrix}a_0 & a_1 & \cdots & a_{t+1} \\
   u_0 & u_1 & \cdots & u_{t+1}. 
  \end{bmatrix}
\]
The column $\begin{bmatrix}k\\m-i\end{bmatrix}$ cannot be the last one in $y$ since $\{a_0, a_1, \dots, a_{t+1}\} = \{0, 1, \dots, n\}$ and $k<n$. Consider the column in $y$ after $\begin{bmatrix}k\\m-i\end{bmatrix}$. The difference between it and $\begin{bmatrix}k\\m-i\end{bmatrix}$ can neither in first nor in second entry exceed one as else $y\in Y^0$. The column cannot equal $\begin{bmatrix}k\\m-i+1\end{bmatrix}$ or $\begin{bmatrix}k\\m-i\end{bmatrix}$ since then $y \in Y^i$ or $y$ degenerate. Thus it has to equal $\begin{bmatrix}k+1\\?\end{bmatrix}$. The second entry cannot be $m-i+1$ since then we could insert between these two columns an entry $\begin{bmatrix}k\\m-i+1\end{bmatrix}$ so that $y$ would be a face of a simplex in $Y^i$, so would itself be in $Y^i$. Therefore, this column has to be $\begin{bmatrix}k+1\\m-i\end{bmatrix}$.

The set of $u$'s in $dy$ equals that in $y$. Thus, $dy\notin Y^0$ as $y\notin Y^0$. Thus, $dy$ can only be in $Y^i[t]$ if it either contains $\begin{bmatrix}k\\m-j\end{bmatrix}$ for some $j\leq i$ or it is the $\begin{bmatrix}k\\m-j\end{bmatrix}$-face of another simplex with $j\leq i-1$. Both is absurd.
                                                                                                                                                                                                 
As $\begin{bmatrix}k\\m-i\end{bmatrix}$ was not the last column, we have indeed proved that the inclusion $Y^i[t] \to Y^i[t+1]$ is a cobase change along inner and left horns, one horn inclusion for each $y$. If $\begin{bmatrix}k\\m-i\end{bmatrix}$ was actually the zeroth column (in the case we are filling a $0$-horn), the edge from the zeroth vertex to the first vertex is marked as the second entries of both agree and the edge is therefore a product of a marked and a degenerate edge. Thus, $\phi$ is in $J$-cell. 
\end{proof}

For the next proposition, please recall the notation $\KK(\DD)$ from the beginning of Section \ref{MCFSection}.

\begin{prop}\label{FibrancyCriteriumProp}The simplicial space $N_\xi\MM$ is Reedy fibrant if $\MM$ has the right lifting property with respect to all
\[c\Sd^2\Lambda^k[n] \to \KK(c\Sd^2\Lambda^k[n])\]
with $0\leq k\leq n$ and $n\geq 1$, where the relative structures on $c\Sd^2\Lambda^k[n]$ and $\KK(c\Sd^2\Lambda^k[n])$ are induced\footnote{This is detailed before Proposition \ref{HolimProp}.} by an arbitrary relative structure on $\underline{n}$ satisfying the following conditions:
\begin{itemize}
\item $\un$ has at least one non-identity weak equivalence, 
\item $0\to 1$ in $\un$ is a weak equivalence if $k=0$,
\item  $n-1 \to n$ in $\un$ is a weak equivalence if $k=n$.
\end{itemize}
\end{prop}
\begin{proof}Fix $0\leq k\leq n$ and equip $\underline{n}$ with a relative structure with at least one weak equivalence such that the map $0\to 1$ in $\underline{n}$ is a weak equivalence if $k=0$ and $(n-1)\to n$ is a weak equivalence if $k=n$. Assume that $\MM$ has the right lifting property with respect to
\[c\Sd^2\Lambda^k[n] \to \KK(c\Sd^2\Lambda^k[n]).\]
By Lemma \ref{LiftingLemma2}, we only need to show that $\MM$ has then also the right lifting property with respect to
\[c\Sd^2\Lambda^k[n] \to \xi\underline{n},\]

We will proceed as in \cite[Lemma 4.2]{M-O14}, but we have to take extra care here since not all morphisms are weak equivalences.

Recall that $\KK(c\Sd^2 \Lambda^k[n])$ is isomorphic to a full subposet $\mathcal{P}_k$ of $c\Sd^2 \Delta[n]$, described as follows: The subposet $c\Sd^2\Lambda^k[n]$ of $c\Sd^2\Delta[n]$ consists of all those sequences $v_0\subsetneq \ldots \subsetneq v_{m}$ for which $v_m \neq \underline{n}$ and $v_m\neq \unn$. The subposet $\mathcal{P}_k$ of $c\Sd^2\Delta[n]$ consists of all sequences $v_0\subsetneq \ldots \subsetneq v_{m}$ in $c\Sd^2\Lambda^i[n]$, for each such sequence also the sequence $v_0\subsetneq \ldots \subsetneq v_{m}\subsetneq \underline{n}$, and finally the sequence consisting only of $\underline{n}$ (corresponding to $k_{c\Sd^2 \Lambda^k[n]}\in\KK(c\Sd^2 \Lambda^k[n]))$.
 
It is enough to show that each relative functor defined on $\mathcal{P}_k$ can be extended to $c\Sd^2 \Delta[n]$. We will give a retraction for the inclusion of $\mathcal{P}_k$ into $c\Sd^2\Delta[n]$, i.e.\ an order-preserving map $c\Sd^2\Delta[n] \to \mathcal{P}_k$, which is the identity on $\mathcal{P}_k$ and respects weak equivalences. This will complete the proof.

Observe that the only objects of $c\Sd^2 \Delta[n]$ which are not in $\mathcal{P}_k$ are given by sequences in which $\unn$ occurs; more precisely, these are the sequences $\unn$, $\unn \subsetneq \un$, and $w_0 \subsetneq \ldots \subsetneq w_{l} \subsetneq \unn $ and $w_0 \subsetneq \ldots \subsetneq w_{l} \subsetneq \unn \subsetneq \un$, where in the last two cases, $w_0 \subsetneq \ldots \subsetneq w_{l}$ is a sequence of non-empty subsets of $\unn$.

The map $r\colon c\Sd^2\Delta[n] \to \mathcal{P}_k$ is described as follows:
\begin{eqnarray*}
 A\mapsto \begin{cases}
           A, &\mbox{ if } A \in \mathcal{P}_k,\\
           \un,& \mbox{ if } A=\unn \mbox{ or } \unn \subsetneq \un, \\
           w_0 \subsetneq w_1\subsetneq \ldots \subsetneq w_l \subsetneq \un, &\mbox{ if } A=(w_0 \subsetneq w_1\subsetneq \ldots \subsetneq w_l \subsetneq \unn) \mbox{ or }\\
           &A=(w_0 \subsetneq w_1\subsetneq \ldots \subsetneq w_l \subsetneq \unn \subsetneq \un).
          \end{cases}
\end{eqnarray*}

Note that the assignment above covers all cases. Furthermore, the map takes only values in $\mathcal{P}_k$, is by definition identity on $\mathcal{P}_k$ and it is checked in \cite[Lemma 4.2]{M-O14} that it is order-preserving. We have only to show that it preserves weak equivalences. As described in Section \ref{BKSection}, weak equivalences are detected by the smallest element of the first set in the chain. This can only change by application of $r$ if $k=0$ and then it can change at most from $0$ to $1$. As the morphism $0\to 1$ is a weak equivalence if $k=0$, the retraction $r$ preserves weak equivalences. 

This completes the proof of the proposition. 
\end{proof}

\section{Contractible Subsets of Simplices}\label{CLSection}
The aim of this section is to prove Lemma \ref{Contractible}. Throughout this section, we use the notation of Section \ref{MCFSection}. This means that $n\geq 1$ is a fixed natural number and $0\leq k\leq n$. Furthermore, $\DD = c\Sd^2 \Lambda^k[n]$ and $\pi\colon \DD \to \un$ is the functor described in Section \ref{MCFSection}. Moreover, $V_+\CC \subset \DD$ denotes for a subcategory $\CC \subset \DD$ the full subcategory of all $d\in\DD$ such that a morphism $c\to d$ in $\DD$ with $c\in\CC$ exists. 

We will split up the statement of Lemma \ref{Contractible} into several lemmas. 

 \begin{lemma} Let $\EE \subset \un$ be a subcategory, not containing $\un \setminus \{k\}$ as connected component. Then the nerve of the category $\pi^{-1}(\EE)$ is weakly equivalent to the nerve of $\EE$. In particular, if $\EE$ is in addition non-empty and connected, the nerve of $\pi^{-1}(\EE)$ is contractible. 
\end{lemma}
\begin{proof}
We can assume that $\EE$ is non-empty and connected. We have to show that $\mathrm{Nerve}\,\pi^{-1}(\EE)$ is contractible. This is clear for $\EE = \un$, so that we can assume $\EE \neq \un$. Throughout this proof, we mean by $W_\bullet$ a non-empty chain $W_0\subsetneq \cdots \subsetneq W_l$ of non-empty subsets of $\un$ such that $W_l$ is neither $\un$ nor $\un\setminus \{k\}$. 

Let $\CC \subset \pi^{-1}(\EE)$ be the full subcategory of all those $W_\bullet$ with $W_0 \subset \EE$. We want to apply Quillen's Theorem A to show that the inclusion $\CC \to \pi^{-1}(\EE)$ induces a weak homotopy equivalence on nerves. We have to show that for every $W_\bullet \in \pi^{-1}(\EE)$, the nerve of the undercategory $W_\bullet/\CC$ is contractible. As $\pi^{-1}(\EE)$ is a poset, we can identify this undercategory with a subcategory of $\CC$. With this identification, there is an adjunction
\[\xymatrix{ \DD \ar@<0.5ex>[r]^-{\lambda} & W_\bullet/\CC \ar@<0.5ex>[l]^-{\rho}, }\]
where $\DD$ is the poset of non-empty chains of non-empty subsets of $W_0\cap \EE$ and $\lambda$ and $\rho$ are defined as follows: We define $\lambda(V_\bullet)$ as the concatenation $V_\bullet \subset W_\bullet$ and we define $\rho(V_\bullet)$ as $V_0\subsetneq \cdots \subsetneq V_q$, where $q\geq 0$ is the largest index such that $V_q\subset W_0\cap \EE$; as $V_\bullet \in \CC$, such an index must exist. By the adjunction, $\mathrm{Nerve}\,\DD \simeq \mathrm{Nerve}\,(W_\bullet/\CC)$. The nerve of $\DD$ is the double-subdivision of a (non-empty) simplex and thus contractible. Thus, $\CC \to \pi^{-1}(\EE)$ induces a weak equivalence on nerves. 

Let $\DD'$ be the full subcategory of $\pi^{-1}(\EE)$ of all chains $W_\bullet$ with $W_l \subset \EE$. There is an obvious inclusion $\DD' \to \CC$. This has a a right adjoint $s\colon \CC \to  \DD'$ with $s(W_\bullet) = W_0 \subsetneq \cdots \subsetneq W_q$ for $q\geq 0$ the largest index with $W_q\subset \EE$. Thus, 
\[\mathrm{Nerve}\, \pi^{-1}(\EE) \simeq \mathrm{Nerve}\, \CC \simeq \mathrm{Nerve}\, \DD'\]
and we only need to show that $\mathrm{Nerve}\,\DD'$ is contractible. As $\EE$ is neither $\un$ nor $\un \setminus \{k\}$, we have $\mathrm{Nerve}\,\DD' \cong \Sd^2 \mathrm{Nerve}\, \EE \simeq \ast$. This completes the proof. 
\end{proof}

Denote $\pi^{-1}(i) \cap V_+ \pi^{-1}(\underline{i-1})$ for $i\geq 1$ for short by $X_i^k$. We want to show that the nerve of $X_i^k$ is contractible unless $k=n$ and $i$ equals $n-1$ or $n$. 

Define $\overline{X}_i^k$ as the poset of chains $W_0\subsetneq \cdots \subsetneq W_l$ in $X_i^k$ such that $W_0 = \{i\}$. There is a left adjoint $\lambda$ to the inclusion $\overline{X}_i^k \to X_i^k$, defined by 
\[\lambda(W_0\subsetneq \cdots \subsetneq W_l) = (\{i\} \subsetneq W_0\subsetneq \cdots \subsetneq W_l)\]
for chains not in $\overline{X}_i^k$ and $\lambda|_{\overline{X}_i^k} = \id_{\overline{X}_i^k}$. Thus, the nerves of $X_i^k$ and $\overline{X}_i^k$ are homotopy equivalent.

\begin{lemma}We have
\[\overline{X}_i^k \cong 
\begin{cases}
c(\Sd^2\partial \Delta[n-1]) \setms c(\Sd^2d_{01\dots(i-1)}\Delta[n-1]) & \text{ for } k= i \\
c(\Sd^2\Lambda^k[n-1]) \setms c(\Sd^2d_{01\dots(i-1)}\Delta[n-1]) & \text{ for } k < i \\
c(\Sd^2\Lambda^{k-1}[n-1]) \setms c(\Sd^2 d_{01\dots(i-1)}\Delta[n-1]) & \text{ for } k > i
\end{cases}\]
Here, $d_{01\dots(i-1)}$ is the face operator induced by the map $[n-1-i] \to [n-1], m\mapsto m+i$; the face is understood to be empty if $i=n$. Furthermore, if $\EE \subset \CC$ is a full subcategory, we denote by $\CC \setminus \EE$ the full subcategory of $\CC$ with objects $\Ob\CC \setminus \Ob\EE$.
\end{lemma}
\begin{proof}
By deleting all occurrences of $i$, we can identify $\overline{X}_i^k$ with the poset of all chains $W_\bullet = (W_0\subsetneq \cdots \subsetneq W_l)$ in $\un$ such that $W_l$ is neither $\un \setminus \{i\}$ nor $\un \setminus \{i,k\}$, $W_l$ contains a $j<i$ and none of the sets contains $i$. 

Denote the poset of all $W_\bullet$ such that $W_l$ is neither  $\un \setminus \{i\}$ nor $\un \setminus \{i,k\}$ and none of the sets contains $i$ by $Y_i^k$. We have 
\[Y_i^k \cong 
\begin{cases}
c\Sd^2\partial \Delta[n-1] & \text{ for } k= i \\
c\Sd^2\Lambda^k[n-1]& \text{ for } k < i \\
c\Sd^2\Lambda^{k-1}[n-1] & \text{ for } k > i
\end{cases}\]
The poset of all chains not containing a $j<i$ forms the subcategory $c\Sd^2 d_{01\dots(i-1)}\Delta[n-1]$. 
\end{proof}

\begin{lemma}
Let $L$ be a (topologically realized) simplicial complex and $K \subset L$ be a \emph{full} subcomplex. This means that any collection $v_0,\dots, v_i$ of $0$-simplices in $K$ spans a simplex in $K$ if it spans a simplex in $L$. Then 
\[L \setminus K \simeq L\setminus \mathrm{st}(K),\]
 where $\mathrm{st}(K)$ is the open star of $K$, the union of all interiors of simplices in $L$ having non-empty intersection with $K$.
\end{lemma}
\begin{proof}
It is enough to find a deformation retraction of $L \setms K$ onto $L\setms \mathrm{st}(K)$. First consider the case $L = \Delta[i]$ and $K = d_{j+1\dots i}\Delta[i] = \Delta[j]$.\footnote{We use here the same symbol for the topological simplex as for the simplex in simplicial sets, but this should not be confusing.} In this case, we use 
\[ (\Delta[i]\setminus \Delta[j]) \times I \to \Delta[i]\setminus \Delta[j],\quad (t_0:\dots : t_i, s) \mapsto (st_0:\dots : st_j: t_{j+1}:\dots : t_i).\]
These are homogeneous coordinates, i.e.\ we implicitly normalize. For a general subsimplex $K$ of $\Delta[i]$, we multiply exactly the homogeneous coordinates corresponding to the possibly non-zero coordinates in $K$ by $s$ instead.

In the general case, it is enough to define the map $(L\setminus K) \times I \to L\setminus K$ on every single (half-open) simplex in a way compatible with restriction to subsimplices. As $K$ is full, the intersection of $K$ with an arbitrary simplex $\Delta$ in $L$ is a subsimplex of $\Delta$. Thus, we can use the map described above. 
\end{proof}

These two lemmas implies the following result: 

\begin{lemma}
The nerve of the category $\pi^{-1}(i) \cap V_+\pi^{-1}(\underline{i-1})$ is contractible for $i\geq 1$ unless $k=n$ and $i$ equals $n-1$ or $n$. 
\end{lemma}
\begin{proof}
Observe first the following two simple facts:
\begin{enumerate}
\item $\mathrm{Nerve}\,c\,\mathrm{Nerve} \cong \mathrm{Nerve}$
\item If $\EE\subset \CC$ is a full subcategory such that $|\mathrm{Nerve} \CC|$ is a simplicial complex, then $|\mathrm{Nerve}(\CC\setminus \EE)| \cong |\mathrm{Nerve}(\CC)| \setminus \mathrm{st}(|\mathrm{Nerve}(\EE)|)$.
\end{enumerate}

From the last two lemmas, we get then
\[|\mathrm{Nerve}\,\overline{X}_i^k| \simeq
\begin{cases}
\partial \Delta[n-1] \setms d_{01\dots(i-1)}\Delta[n-1] & \text{ for } k= i \\
\Lambda^k[n-1| \setms d_{01\dots(i-1)}\Delta[n-1] & \text{ for } k < i \\
\Lambda^{k-1}[n-1] \setms d_{01\dots(i-1)}\Delta[n-1] & \text{ for } k > i.
\end{cases}\]
Denote the right hand side by $Z_i^k$.

In the case $k=i$, we remove from $\partial \Delta[n-1] \cong S^{n-2}$ a $\Delta[n-i-1]$. Thus, $Z_i^k$ is contractible unless $k = i=n$. Indeed, choose a vertex $j$ in $\Delta[n-i-1]$. Then there is a deformation retraction of $Z_i^k$ onto $d_j \Delta[n-1]$, with possibly a subsimplex removed; this is contractible. 

In the case $k\neq i$, the simplex $d_{01\dots(i-1)}\Delta[n-1]$ neither equals the vertex $k$ for $k<i$ nor $k-1$ for $k>i$ (as $d_{01\dots(i-1)}\Delta[n-1]$ is just one point only if $i = n-1$). If the tip of the horn is not in $d_{01\dots(i-1)}\Delta[n-1]$, a linear deformation towards the tip is the required contraction of $Z_i^k$. If the tip of the horn is in $d_{01\dots(i-1)}\Delta[n-1]$, deforming away from the tip gives a homotopy equivalence of $Z_i^k$ to $\partial \Delta[n-2]$ with one (non-empty) subsimplex removed. The same argument as before gives that this is contractible. 
\end{proof}

It remains to show the following lemma: 
\begin{lemma}If $k=n\geq 2$, then the nerve of $Y = \pi^{-1}((n-1)\to n) \cap V_+\pi^{-1}(\underline{n-2})$ is contractible. 
\end{lemma}
\begin{proof}
Let $W_\bullet = (W_0\subsetneq \cdots \subsetneq W_l)$. Then $W_\bullet\in Y$ if and only if 
\begin{itemize}
\item $W_0$ contains no elements but $(n-1)$ and $n$,
\item $W_l \neq \un$ and $W_l \neq \underline{n-1}$,
\item there exists $j\in W_l$ with $j<n-1$.
\end{itemize} 
Denote the subcategory of those $W_\bullet$ in $Y$ with $W_0 = \{n\}$ by $Y_0$. By deleting $n$, this can be identified with $c\Sd^2\partial \Delta[n-1]\setms\{n-1\}$. 

Denote the subcategory of those $W_\bullet$ in $Y$ with $W_0 = \{n-1\}$ by $Y_1$. By deleting $(n-1)$, this can be identified with $c\Sd^2\Lambda^{n-1}[n-1]\setms\{n-1\}$.

Denote the subcategory of those $W_\bullet$ in $Y$ with $W_0$ or $W_1$ equaling $\{n-1,n\}$ by $Y_2$. This is isomorphic to $(c\Sd^2 \partial \Delta[n-2])\times (c\Sd \Delta[1])$, where the second coordinate corresponds to $W_0$ being $\{n\}$, $\{n-1, n\}$ or $\{n-1\}$. 

The intersection of $Y_0$ and $Y_2$ is given by all the $W_\bullet$ in $Y$ with $W_0 = \{n\}$ and $W_1 = \{n-1,n\}$. By deleting all entries of the form $(n-1)$ and $n$, the intersection $Y_0\cap Y_2$ is isomorphic to $c\Sd^2 \partial \Delta[n-2]$. 

By a similar argument $Y_1\cap Y_2 \cong c\Sd^2 \partial \Delta[n-2]$.

In total, we see that 
\[\mathrm{Nerve}\, Y = \mathrm{Nerve}\, Y_0 \cup_{\Sd^2 \partial \Delta[n-2]} \mathrm{Nerve}\, Y_2 \cup_{\Sd^2 \partial \Delta[n-2]} \mathrm{Nerve}\, Y_1.\] By the identifications above, this is after geometric realization homeomorphic to  
\[D^{n-2} \cup_{S^{n-3}} S^{n-3}\times I \cup_{S^{n-3}} S^{n-3} \times I.\]
This colimit in turn is homeomorphic to $D^{n-2}$, which is contractible. 
\end{proof}

\bibliographystyle{alpha}
\bibliography{../Chromatic}
\end{document}